\documentclass[12pt]{article}

\usepackage{fullpage,amsmath,amssymb,amsthm}

\newtheorem{theorem}{Theorem}[section]
\newtheorem{conjecture}[theorem]{Conjecture}
\newtheorem{lemma}[theorem]{Lemma}
\newtheorem{proposition}[theorem]{Proposition}
\newtheorem{prop}[theorem]{Proposition}
\newtheorem{question}[theorem]{Question}
\newtheorem{corollary}[theorem]{Corollary}
\newtheorem{observation}[theorem]{Observation}

\newcommand{\F}{\mathcal{F}}
\newcommand{\A}{\mathcal{A}}
\newcommand{\B}{\mathcal{B}}
\newcommand{\K}{\mathcal{K}}
\newcommand{\T}{\mathcal{T}}
\newcommand{\X}{\mathcal{X}}
\newcommand{\Y}{\mathcal{Y}}

\renewcommand{\P}{\mathbb{P}}
\newcommand{\E}{\mathbb{E}}

\newcommand{\ep}{\varepsilon}

\renewcommand{\bar}[1]{\overline{#1}}

\renewcommand{\le}{\leqslant}
\renewcommand{\ge}{\geqslant}
\renewcommand{\leq}{\leqslant}
\renewcommand{\geq}{\geqslant}

\newcommand{\loga}{\log\frac{\binom{n-1}{k}}{a}}
\newcommand{\logaF}{\log\frac{\binom{n-1}{k}}{a_\F}}
\newcommand{\eoaloga}{\exp\left( o\left( a \loga \right) \right)}

\title{Sharp threshold for the Erd\H{o}s--Ko--Rado theorem}
\author{
J\'ozsef Balogh\footnote{Department of Mathematics, University of Illinois at Urbana-Champaign, Urbana, Illinois 61801, USA, and Moscow Institute of Physics and Technology, Russian Federation. E-mail: \texttt{jobal@illinois.edu}. Research supported by NSF RTG Grant DMS-1937241, NSF Grant DMS-1764123, Arnold O. Beckman Research Award (UIUC Campus Research Board RB 18132), the Langan Scholar Fund (UIUC), and the Simons Fellowship.}
\and
Robert A.~Krueger\footnote{Department of Mathematics, University of Illinois at Urbana-Champaign, Urbana, Illinois 61801, USA. Email: \texttt{rak5@illinois.edu}. Research partially supported by NSF RTG Grant DMS-1937241 and UIUC Campus Research Board Award RB 18132.}
\and
Haoran Luo\footnote{Department of Mathematics, University of Illinois at Urbana-Champaign, Urbana, Illinois 61801, USA. Email: \texttt{haoranl8@illinois.edu}.}}
\date{}

\begin{document}
\maketitle
\begin{abstract}
For positive integers $n$ and $k$ with $n\geq 2k+1$, the Kneser graph $K(n,k)$ is the graph with vertex set consisting of all $k$-sets of $\{1,\dots,n\}$, where two $k$-sets are adjacent exactly when they are disjoint. The independent sets of $K(n,k)$ are $k$-uniform intersecting families, and hence the maximum size independent set is given by the Erd\H{o}s-Ko-Rado Theorem. Let $K_p(n,k)$ be a random spanning subgraph of $K(n,k)$ where each edge is included independently with probability $p$. Bollob\'as, Narayanan, and Raigorodskii asked for what $p$ does $K_p(n,k)$ have the same independence number as $K(n,k)$ with high probability. For $n=2k+1$, we prove a hitting time result, which gives a sharp threshold for this problem at $p=3/4$. Additionally, completing work of Das and Tran and work of Devlin and Kahn, we determine a sharp threshold function for all $n>2k+1$.
\medskip


\end{abstract}

\section{Introduction} \label{sec::int}

For a positive integer $n$, we denote by $[n]$ the set $\{1,2,\dots,n\}$. For a set $A$, we denote by $\binom{A}{k}$ (resp.~$\binom{A}{\leq k}$) the collection of all subsets of $A$ of size (resp.~at most) $k$.

For positive integers $n$ and $k$, the \emph{Kneser graph} $K(n,k)$ is the graph on $\binom{[n]}{k}$, where two sets are adjacent exactly when they are disjoint. This graph has no edges unless $n\geq 2k$, $K(2k,k)$ is just a matching, and $K(n,1)$ is a complete graph, so we assume from now on that $n\geq 2k+1$ and $k \geq 2$. The independent sets of $K(n,k)$ are $k$-uniform intersecting families, important objects of research in extremal combinatorics. For $x \in [n]$, we call $\K_x = \{S \in \binom{[n]}{k} : x \in S\}$ the \emph{star} based at $x$. Note that $\K_x$ is an independent set of $K(n,k)$ of size $\binom{n-1}{k-1}$. The classical Erd\H{o}s-Ko-Rado theorem~\cite{EKR} states that $\alpha(K(n,k)) = \binom{n-1}{k-1}$ when $n\geq 2k$, where $\alpha(G)$ is the maximum size of an independent set in $G$. An extension by Hilton and Milner~\cite{HM} implies that every maximum independent set of the Kneser graph is a star, when $n\geq 2k+1$.

Let $K_p(n,k)$ be a random spanning subgraph of $K(n,k)$, where each edge is included independently with probability $p = p(n,k)$. If $E = E(n,k)$ is an event defined on the probability space associated with $K_p(n,k)$, we say $E$ occurs \emph{with high probability} (\emph{w.h.p.\@}) if the probability of $E$ tends to $1$ as $n \to \infty$. In general, we consider $k$ to depend on $n$, so we need only take $n\to\infty$. Bollob\'as, Narayanan, and Raigorodskii~\cite{BNR} proposed a random variant of the Erd\H{o}s-Ko-Rado theorem, following a recent trend of so-called transference results. 
\begin{question}[\cite{BNR}]\label{q::main}
For what $p$ does $\alpha(K_p(n,k)) = \alpha(K(n,k))$ with high probability?
\end{question}
Informally, this is asking if the Erd\H{o}s-Ko-Rado theorem still holds if we `forget' that a randomly selected portion of disjoint pairs are in fact disjoint. Since the probability that $\alpha(K_p(n,k)) = \binom{n-1}{k-1}$ increases with $p$, to answer Question~\ref{q::main}, one wants to determine the `smallest' $p$ for which $\alpha(K_p(n,k)) = \binom{n-1}{k-1}$ w.h.p.

An obvious necessary condition to $\alpha(K_p(n,k)) = \binom{n-1}{k-1}$ is the maximality of stars as independent sets. A \emph{superstar based at $x$} is a set family of the form $\{A\} \cup \K_x$, where $A \not\in \K_x$. Thus $\alpha(K_p(n,k)) > \binom{n-1}{k-1}$ if there exists a superstar spanning an independent set. The expected number of independent superstars in $K_p(n,k)$ is
\[ n \binom{n-1}{k} (1-p)^{\binom{n-k-1}{k-1}} ,\]
which changes from $o(1)$ to $\omega(1)$ when $p \approx p_0$, where
\begin{equation}\label{eq::p0}
p = p_0 :=
\begin{cases}
3/4, & \textrm{if $n = 2k+1$}, \\
\log\left(n\binom{n-1}{k}\right)/\binom{n-k-1}{k-1}, & \textrm{if $n > 2k+1$},
\end{cases}
\end{equation}
and $\log$ denotes the natural logarithm. The only reason that $n=2k+1$ seems different is because we lose the approximation $1-p_0 \approx e^{-p_0}$ in that case. A straightforward second moment calculation (see \cite{BNR}, \cite{DT}, or Observation~\ref{obs::thresh}) gives that for every constant $\ep>0$, for every $p \leq (1-\ep)p_0$, $K_p(n,k)$ has an independent superstar w.h.p., and thus $\alpha(K_p(n,k)) > \binom{n-1}{k-1}$ w.h.p.

When they posed Question~\ref{q::main}, Bollob\'as, Narayanan, and Raigorodskii~\cite{BNR} effectively answered it for $k = o(n^{1/3})$. They showed that for all constant $\ep>0$, for all $p \geq (1+\ep)p_0$, $\alpha(K_p(n,k)) = \binom{n-1}{k-1}$ w.h.p., showing that $p_0$ is a `sharp threshold function' for this problem when $k = o(n^{1/3})$. Subsequently, Balogh, Bollob\'as, and Narayanan~\cite{BBN} made progress on Question~\ref{q::main} for all $k \leq (\frac{1}{2}-\gamma)n$, for every constant $\gamma>0$. Das and Tran~\cite{DT} formulated a removal lemma for $K(n,k)$, showed that $p_0$ is a sharp threshold function for $k \leq n/C$, for some constant $C$, and showed that $p_0$ is a `coarse' threshold function for $k \leq (\frac{1}{2}-\gamma)n$, for every constant $\gamma>0$, meaning that there exists some constant $c=c(\gamma)>1$ such that for all $p \geq cp_0$, $\alpha(K_p(n,k)) = \binom{n-1}{k-1}$. Concurrently, Devlin and Kahn~\cite{DK} showed that $p_0$ is a coarse threshold function for all $n \geq 2k+1$, that is, they removed the dependence of $c$ on $\gamma$. In particular, they made the first progress on Question~\ref{q::main} for $n=2k+1$, showing that for some $\ep>0$, $\alpha(K_{(1-\ep)}(2k+1,k)) = \binom{2k}{k-1}$ w.h.p. What remains is to show that $p_0$ is a sharp threshold for $k \geq n/C$, which is what we do in the present work.

In fact, we can prove something stronger, as most previous papers have done: above a threshold, not only are stars maximum independent sets, but they are unique. We say that a spanning subgraph of $K(n,k)$ is \emph{EKR} if every maximum independent set is a star. Clearly, if $K_p(n,k)$ is EKR w.h.p., then $\alpha(K_p(n,k)) = \binom{n-1}{k-1}$ w.h.p. Thus, to resolve Question~\ref{q::main}, it suffices to show that for $p \geq (1+\ep)p_0$, $K_p(n,k)$ is EKR w.h.p.
\begin{theorem}\label{thm::main}
Let $\ep>0$ be a constant. If $n \geq 2k+1$ and $p \geq (1+\ep)p_0$, then $K_p(n,k)$ is EKR w.h.p.
\end{theorem}
Perhaps Theorem~\ref{thm::main} does not completely resolve Question~\ref{q::main} because of the $\ep$. Bollob\'as, Narayanan, and Raigorodskii~\cite{BNR} raise the natural question of determining the `width' of the sharp threshold for Question~\ref{q::main}: roughly, how small, in terms of $n$ and $k$, may we take $\ep$ in Theorem~\ref{thm::main}? Das and Tran~\cite{DT} proved Theorem~\ref{thm::main} for all $\ep \gg \frac{1}{k}$ and $k \leq n/C$. Such tight control of $\ep$ and the definition of $p_0$ indicate that in some sense the independence of superstars is the most difficult obstacle to overcome for $\alpha(K_p(n,k))$ to equal $\binom{n-1}{k-1}$. Das and Tran~\cite{DT} suggested a stronger, hitting time version of Question~\ref{q::main} to capture this notion.

To give the statement, we need to consider the \emph{random subgraph process}. Consider a graph $G$, and let $e_1, \dots, e_m$ be the edges of $G$. Given a permutation $\pi: E(G) \to [m]$, we define $G_i(\pi)$ to be the spanning subgraph of $G$ with edges $e_k$ satisfying $\pi(e_k) \leq i$. Note that $G_i(\pi)$ has exactly $i$ edges. If we select $\pi$ uniformly at random, we can think of the sequence $G_1(\pi), G_2(\pi), \dots, G_m(\pi)$ as adding one edge of $G$ picked uniformly at random to $V(G)$ until there are no more edges to be added. For a monotone property $P$ which $G$ satisfies, we define the \emph{hitting time} $\tau_P$ to be the minimum $i$ such that $G_i$ satisfies property $P$. Note that $\tau_P$ is a random variable depending on the permutation $\pi$, and $G_{\tau_P}$ is a random subgraph of $G$ satisfying property $P$, but $G_{\tau_P}$ is not necessarily chosen uniformly from all such subgraphs.

For $n=2k+1$, we prove that as soon as stars are maximal independent sets, stars are maximum independent sets w.h.p., where the `with high probability' is in the context of the underlying probability space of the random subgraph process. To be precise, let
\[ \tau_{\text{super}} = \min\{i : K_i(n,k) \text{ has no independent superstars}\} ,\]
\[ \tau_{\alpha} = \min\left\{i : \alpha(K_i(n,k)) = \binom{n-1}{k-1}\right\} .\]
\begin{theorem}\label{thm::hitting:indep}
With high probability, $\tau_{\alpha} = \tau_{\textnormal{super}}$ when $n=2k+1$.
\end{theorem}
This theorem reduces Question~\ref{q::main} to: for what $p$ is it true that $K_p(n,k)$ has no independent superstars? We defined $p_0$ to answer this question, so with an easy first moment calculation (see Observation~\ref{obs::threshupper}), Theorem~\ref{thm::hitting:indep} gives the sharp threshold for $\alpha(K_p(2k+1,k)) = \binom{2k}{k-1}$.

Analogously, we can prove a hitting time theorem for the EKR property, the property that stars are the unique maximum independent sets. A \emph{near-star based at $x$} is a set family of the form $\{A\} \cup (\K_x \setminus \{B\})$, where $A \not\in \K_x$ and $B \in \K_x$. The obvious necessary condition to $K_p(n,k)$ being EKR is that there are no independent near-stars. Let
\[ \tau_{\text{near}} = \min\{i : K_i(n,k) \text{ has no independent near-stars}\} ,\]
\[ \tau_{\text{EKR}} = \min\{i : K_i(n,k) \text{ is EKR}\} .\]
\begin{theorem}\label{thm::hitting:EKR}
With high probability, $\tau_{\textnormal{EKR}} = \tau_{\textnormal{near}}$ when $n=2k+1$.
\end{theorem}
In fact, $p_0$ is a sharp threshold function for both lacking independent superstars and lacking independent near-stars, so Theorem~\ref{thm::hitting:EKR} proves Theorem~\ref{thm::main} for $n=2k+1$. But since $\tau_{\text{super}} < \tau_{\text{near}}$ deterministically, neither Theorem~\ref{thm::hitting:indep} nor Theorem~\ref{thm::hitting:EKR} implies the other.

A full resolution of Question~\ref{q::main} would be the proof of hitting time theorems for all $n$ and $k$. We believe this ought to be true.
\begin{conjecture}\label{conj::hitting}
With high probability, for all $n>2k+1$, $\tau_\alpha = \tau_{\textnormal{super}}$ and $\tau_{\textnormal{near}} = \tau_{\textnormal{EKR}}$.
\end{conjecture}
Our methods for Theorems~\ref{thm::hitting:EKR} and~\ref{thm::hitting:indep} extend easily to when $n-2k = o(n)$, and a careful analysis of \cite{BNR,DT} proves Conjecture~\ref{conj::hitting} when $k = o(n)$. To name a particular case of Conjecture~\ref{conj::hitting} which we do not know how to solve, consider $n=4k$. See Section~\ref{sec::hit:proof} for some technical remarks on Conjecture~\ref{conj::hitting}.

Having fully addressed Question~\ref{q::main}, we now propose an extension of Theorems~\ref{thm::hitting:indep} and~\ref{thm::hitting:EKR}, and we mention other modern extensions of the Erd\H{o}s-Ko-Rado theorem. With our techniques, one should be able to obtain finer information about the independent sets of $K_p(n,k)$ for arbitrary $p$, possibly answering the following generalization of Question~\ref{q::main}: what is the size of the largest independent set of $K_p(n,k)$ not contained in a star? Conjecture~\ref{conj::hitting} determines when the size of the largest independent set not contained in a star is at most $\binom{n-1}{k-1}-\ell+1$, for $\ell\in\{1,2\}$. The Hilton-Milner theorem~\cite{HM} states that the largest independent set of $K(n,k)$ not contained in a star has size $\binom{n-1}{k-1} - \binom{n-k-1}{k-1} + 1$. Thus we can formulate a generalization of Conjecture~\ref{conj::hitting} as follows. Define
\[ \tau_\ell = \min\{i : \text{ for every $x \in [n]$ and $A \not\in \K_x$, $d_{K_i(n,k)}(A,\K_x) \geq \ell$} \} ,\]
where $d_G(v,S)$ is the number of neighbors in $G$ of $v$ in $S \subseteq V(G)$. Note that $\tau_1 = \tau_{\text{super}}$, $\tau_2 = \tau_{\text{near}}$, and for $\ell = \binom{n-k-1}{k-1}$, $\tau_\ell = |E(K(n,k))|$ always. In this way, the $\tau_\ell / |E(K(n,k))|$ fill the space between $p_0$ and $1$.
\begin{question}
Does the largest independent set of $K_{\tau_\ell}(n,k)$ not contained in a star have size $\binom{n-1}{k-1} - \ell + 1$ with high probability?
\end{question}
We hesitate to call this a conjecture, because while presumably our techniques would work for small $\ell$, the situation for large $\ell$ is less clear.

Our proof method does not require too much structural information about $K(n,k)$; in fact, all that is needed is some regularity in the degrees, that the maximum independent sets are `spread out' in some sense, and some edge-isoperimetry, which is some upper bound on the number of edges a set of vertices can contain. The Kneser graph is a special class of so-called distance graphs, where similar problems have been studied~\cite{Dist}, so it is likely our methods may be applied there as well. Arguably the most studied distance graph is the hypercube. There are many results concerning various graph parameters of the \mbox{(edge-)random} subgraph of the hypercube; see~\cite{CubeSurv} for a comprehensive, but slightly outdated, survey. After writing this paper, we learned that some of those results use a method similar to our Lemma~\ref{lem::Ahigh}. In particular, Kostochka~\cite{Kos} gives a result similar to our Lemma~\ref{lem::Ahigh} but in greater generality.

A more difficult, but natural, random variant of the Erd\H{o}s-Ko-Rado theorem was proposed by Balogh, Bohman, and Mubayi~\cite{BBM}: instead of taking a random spanning subgraph of $K(n,k)$, we take a random induced subgraph of $K(n,k)$, including each vertex with probability $p$. For what $p$ is a maximum independent set of this random induced subgraph contained in a star? This is more naturally stated in the hypergraph version of the Erd\H{o}s-Ko-Rado theorem, where we retain hyperedges with probability $p$. Balogh, Bohman, and Mubayi~\cite{BBM} proved a number of results, mostly for $k \leq n^{1/2-\ep}$. This is a particularly difficult problem because the property of interest is not monotone; for small $p$ and large $p$ the property holds, but sometimes for some $p$ in between, the property does not hold. Hamm and Kahn improved the results for $n^{1/3} \ll k \leq \frac{1}{2}\sqrt{n\log n}$~\cite{HKI} and $n=2k+1$~\cite{HKII}. Gauy, H\`an and Oliveira~\cite{GHO} extended~\cite{BBM} and gave the asymptotic size of largest intersecting family for all $k$ and almost all $p$. Balogh, Das, Delcourt, Liu, and Sharifzadeh~\cite{BDDLS} gave additional results up to $k \leq n/4$, but generally not as tight as~\cite{HKI}. Hamm and Kahn~\cite{HKII} optimistically conjecture that, just like in the problem addressed in this paper, the maximum independent sets are contained in stars (roughly) when the stars induce maximal independent sets.

Balogh, Das, Delcourt, Liu, and Sharifzadeh~\cite{BDDLS} also study the enumeration variant of the Erd\H{o}s-Ko-Rado theorem, determining the order of magnitude of the log of the number of independent sets in $K(n,k)$. Balogh, Das, Liu, Sharifzadeh, and Tran~\cite{BDLST} and independently Kupavskii and Frankl~\cite{FK} strengthened this to say that most independent sets of $K(n,k)$ are contained in stars for $n \geq 2k+c\sqrt{k\log k}$, where $c$ is a large constant in~\cite{BDLST} and $c=2$ in~\cite{FK}. Balogh, Garcia, Li, and Wagner~\cite{BGLW} push this $c\sqrt{k\log k}$ down to $100 \log k$. They conjecture it is true down to $n\geq 2k+2$, since a simple computation shows that the families which show the tightness of the Hilton-Milner theorem~\cite{HM} outnumber the trivial families for $n=2k+1$.

Inspired by Lov\'asz's determination~\cite{LovaszKneser} of the chromatic number of $K(n,k)$ and subsequent work, Kupavskii~\cite{KupKneser} studied the chromatic number of $K_p(n,k)$, giving bounds for a wide range of $n$, $k$ and $p$. The methods used are topological and not related to ours.

The rest of the paper is organized as follows. In Section~\ref{sec::prelim} we collect results from prior works that we use in our proofs. In Section~\ref{sec:main lemma} we prove a lemma on counting sets of vertices of $K(n,k)$, in a similar spirit to the container method~\cite{BMS,ST}. We prove Theorems~\ref{thm::hitting:indep} and \ref{thm::hitting:EKR} in Section~\ref{sec::hit} almost simultaneously, definitively answering the $n=2k+1$ case of Question~\ref{q::main}. Section~\ref{sec::hit:proof} contains some technical remarks concerning Conjecture~\ref{conj::hitting}. In Section~\ref{sec::largen}, we prove Theorem~\ref{thm::main} for $2k+1 < n \leq Ck$, where $C$ is a large constant, the other cases being proven by Theorem~\ref{thm::hitting:EKR} or~\cite{BNR,DT}.

\subsection{Standard estimates and notation}\label{sec::intro:estimates}

All logarithms are base $e$. We make extensive use of standard asymptotic notation to simplify our calculations. We say $f(n) = O(g(n))$ (resp.\ $f(n) = \Omega(g(n))$) if there exists $C>0$ such that $f(n) \leq C g(n)$ (resp.\ $f(n) \geq Cg(n)$) for all $n$ sufficiently large. If $f(n) = O(g(n))$ and $f(n) = \Omega(g(n))$, then we say $f(n) = \Theta(g(n))$. We say $f(n) = o(g(n))$ (resp.\ $f(n) = \omega(g(n))$) if $f(n)/g(n) \to 0$ (resp.\ $\infty$) as $n\to\infty$. Since there is only one variable, $n$, tending to infinity, with the other variables being clearly dependent or clearly independent of $n$, we find the asymptotic notation unambiguous. Still, we use the notation judiciously. If desired, one could eliminate the use of asymptotic notation from this paper entirely, being explicit throughout with constant or logarithmic factors.

We often make use of the standard bound $(\frac{r}{s})^s \leq \binom{r}{s} \leq (\frac{er}{s})^s$ in the weaker form
\[ \binom{r}{s}, \binom{r}{\leq s} = \exp\left( \Theta\left( s\log\frac{r}{s} \right) \right) \]
for $s \leq .99r$. We also use the following version of the Chernoff bound. We use $\P(E)$ to denote the probability of the event $E$ and $\E[X]$ to denote the expected value of a random variable $X$.
\begin{lemma}[Chernoff Bound]\label{lem::cher}
If $X$ is binomially distributed with mean $\mu$, then for $0 \le \alpha \le 1$ and $1<\beta$, we have the lower tail bound
\begin{equation}
\P(X\le \alpha \mu) \leq \exp\left( -(1-\alpha + \alpha\log\alpha) \mu \right) \le \exp\left(-(1-2\sqrt{\alpha})\mu\right)
\end{equation}
and the upper tail bound
\begin{equation}
\P(X\ge \beta \mu) \leq \exp\left( -(1-\beta+\beta\log\beta) \mu \right) \le \exp\left(- \mu \beta \log (\beta/e) \right) .
\end{equation}
\end{lemma}

\section{Preliminaries}\label{sec::prelim}

We make the dependence of $k$ on $n$ explicit wherever possible, and we assume that $n$ is sufficiently large for our calculations to go through. In Section~\ref{sec::hit}, we consider only $n=2k+1$, and in Section~\ref{sec::largen} we consider only $2k+1 < n \leq Ck$, where $C$ is a sufficiently large constant, since Theorem~\ref{thm::main} was already proven for $n>Ck$ in~\cite{BNR,DT}. Despite this, we try to make statements not assuming any relationship between $k$ and $n$ so that our intermediate results are as useful as possible.

Let $\T = \{\F \subseteq V(K(n,k)) : |\F| = \binom{n-1}{k-1}, \text{ $\F$ is not a star}\}$. By the union bound, the probability that $K_p(n,k)$ is not EKR is at most
\begin{equation}\label{eq::unionbound}
\sum_{\F\in\T} (1-p)^{e(\F)} .
\end{equation}
We are done if we can show that \eqref{eq::unionbound} is $o(1)$ for the desired $p$. Unfortunately, $\T$ appears to be too large for this strategy to give a sharp threshold. Our strategy is to use the lower bounds on $e(\F)$ from~\cite{DT,DK}, refine $\T$, and give improved bounds on the size of the refinement to make this strategy successful.

First, we introduce a framework which appeared in~\cite{DK}, although similar ideas appear in prior papers. Consider $\F \subseteq V(K(n,k))$ of size $\binom{n-1}{k-1}$. Let $x_\F$ be the smallest $x \in [n]$ minimizing $|\F \setminus \K_x|$. For ease of notation here, we write $x$ for $x_\F$. Let $\A_\F = \F \setminus \K_x$, and $a_\F = |\A_\F|$. This $a_\F$ measures the `distance' from $\F$ to the nearest star.\footnote{In~\cite{Kupavskii}, $a_\F$ is called the \emph{diversity} of $\F$. The earliest extensions of the Erd\H{o}s-Ko-Rado theorem, for example~\cite{HM}, gave the maximum size of independent $\F \subseteq V(K(n,k))$ satisfying restrictions on $a_\F$.} Let $\B_\F = \K_x \setminus \F$, so that $|\B_\F| = a_\F$ since $|\F| = |\K_x|$. Note that $\A_\F \subseteq \binom{[n]\setminus \{x\}}{k}$ and $\B_\F \subseteq \K_x$; by $\bar{\A}_\F$ and $\bar{\B}_\F$ we mean $\bar{\A}_\F = \binom{[n]\setminus \{x\}}{k} \setminus \A_\F$ and $\bar{\B}_\F = \K_x \setminus \B_\F = \K_x \cap \F$. Since $\bar\B_\F$ is an intersecting set system, $E(\F)$, the set of edges spanned by $\F$, is partitioned into $E(\A_\F)$ and $E(\A_\F, \bar\B_\F)$, the set of edges with one endpoint in $\A_\F$ and the other in $\bar\B_\F$. We mostly focus on the bipartite structure between $\A_\F$ and $\B_\F$,\footnote{It can be helpful to see the hypercube in disguise in these definitions. If one takes the complements of every set of $\A_\F$ in $[n]\setminus\{x\}$, and one also removes the element $x$ from every set of $\B_\F$, then the disjointness relation between $\A_\F$ and $\B_\F$ becomes the subset-superset relation with ground set $[n]\setminus\{x\}$.} but whenever we discuss neighborhoods, denoted by $N(\A)$ for $\A \subseteq V(K(n,k))$, and edges, with $e(\A) = |E(\A)|$ and $e(\A,\A') = |E(\A,\A')|$, the context is always $K(n,k)$ unless explicitly stated otherwise. To summarize,
\[ \F = \A_\F \cup (\K_{x_\F} \setminus \B_\F) ,\]
where $x_\F$ is chosen to minimize $|\A_\F| = |\B_\F|$. Note that if $a_\F = 0$, then $\F$ is a star, and if $a_\F = 1$, then $\F$ is a near-star. We drop the $\F$ subscript when it is clear from context.

We begin with three helpful observations about the above framework. First, by a simple averaging argument, we observe that for every $\F \in \T$, $a_\F$ is not too large.
\begin{observation}\label{obs::abound}
For every $\F \in \T$, $a_\F \leq \frac{n-k}{n} \binom{n-1}{k-1}$.
\end{observation}

\begin{proof}
Note that $\sum_{x\in[n]} |\F\cap \K_x| = k|\F|$ for any $\F$. Thus for some $x$ we have
\[ |\F \setminus \K_x| = \binom{n-1}{k-1} - |\F \cap \K_x| \leq \frac{n-k}{n} \binom{n-1}{k-1} .\qedhere\]
\end{proof}

Second, unless $a_\F$ is very large, $\F$ is not close to $\K_y$ for any $y \neq x_\F$.
\begin{observation}\label{obs::farstar}
For every $\F \in \T$ and $y \in [n]$ with $y \neq x_\F$, we have
\[ |\F \setminus \K_y| \geq \binom{n-2}{k-1} - |\F \setminus \K_x| .\]
\end{observation}

\begin{proof}
Observe that $|\F \setminus \K_y| \geq |(\K_x \cap \F) \setminus \K_y| \geq |\K_x \setminus \K_y| - |\K_x \setminus \F| = \binom{n-2}{k-1} - |\F \setminus \K_x|$.
\end{proof}

Third, we relate $e(\bar\A_\F,\B_\F)$ to $e(\A_\F,\bar\B_\F)$, the latter of which is more directly related to $e(\F)$.

\begin{observation}\label{obs::AbarB}
For every $\F \in \T$, $e(\bar\A_\F,\B_\F) = \binom{n-k-1}{k}a_\F + e(\A_\F,\bar\B_\F)$.
\end{observation}

\begin{proof}
The degree of vertices in $\binom{[n]\setminus\{x\}}{k}$ to $\K_x$ is $\binom{n-k-1}{k-1}$, while the degree of vertices in $\K_x$ to $\binom{[n]\setminus\{x\}}{k}$ is $\binom{n-k}{k}$. Thus
\[ e(\bar\A_\F,\B_\F) = \binom{n-k}{k}a_\F - e(\A_\F,\B_\F) = \binom{n-k}{k} a_\F - \binom{n-k-1}{k-1} a_\F + e(\A_\F,\bar\B_\F) \]
\[ = \binom{n-k-1}{k} a_\F + e(\A_\F,\bar\B_\F) .\qedhere\]
\end{proof}

Let
\[ \T_x(a) = \{\F \in \T : x_\F = x,\, a_\F = a \} .\]
We define subsets $\T^i$ of $\T$ for $i\in [5]$ as is convenient. Whenever $\T^i$ is defined, we define $\T^i_x(a)$ to be $\T^i \cap \T_x(a)$. For reference, we list the definitions of all the $\T^i$ here, along with the subsection where they first appear; we introduce them formally when we need them, so some notation is not yet defined:
\begin{flalign*}
\T^1 &= \left\{ \F \in \T : e(\F) \leq \frac{5}{p_0} a_\F \log\frac{\binom{n-1}{k}}{a_\F} \right\} ,\phantom{\Bigg)} & \tag{Section~\ref{sec::prelim:fewef}} \\
\T^2 &= \left\{ \F \in \T^1 : \B_\F \subseteq N(\A_\F) \right\} ,\phantom{\Bigg)} & \tag{Section~\ref{sec::prelim:neigh}} \\
\T^3 &= \left\{ \F \in \T^2 : \F \text{ is 2-linked} \right\} ,\phantom{\Bigg)} & \tag{Section~\ref{sec::hit:2link}} \\
\T^4 &= \left\{ \F \in \T^1 : \left| A_\F^{<1/\sqrt{k}} \right| \leq \frac{a_\F}{\log\log k} \right\} ,\phantom{\Bigg)} & \tag{Section~\ref{sec::largen:Alow}} \\
\T^5 &= \left\{ \F \in \T^4 : e(\A_\F) \leq \frac{1}{2} e(\F) \right\} .\phantom{\Bigg)} & \tag{Section~\ref{sec::largen:eA}}
\end{flalign*}
As an example of how these definitions are used, in Lemma~\ref{lem::bigeF}, we easily show that no $\F \in \T \setminus \T^1$ is independent in $K_p(n,k)$ for $p \geq p_0$ w.h.p., because by definition, the $\F \in \T \setminus \T^1$ have many edges. In other cases, like Proposition~\ref{prop::T2reduc}, we show that if some $\F \in \T^1$ is independent in $K_p(n,k)$, then some $\F \in \T^2$ is independent in $K_p(n,k)$, which reduces the problem of showing that no $\F \in \T^1$ is independent to showing that no $\F \in \T^2$ is independent.

\subsection{Bounding $e(\F)$}\label{sec::prelim:fewef}

Since $\F$ is determined by $x = x_\F$, $\A_\F \subseteq \binom{[n]\setminus\{x\}}{k}$, and $\B_\F \subseteq \K_x$, we trivially have
\begin{equation}\label{eq::trivialF}
|\T_x(a)| \leq \binom{\binom{n-1}{k}}{a} \binom{\binom{n-1}{k-1}}{a} \leq \binom{\binom{n-1}{k}}{a}^2 .
\end{equation}

To obtain a threshold for Question~\ref{q::main}, Devlin and Kahn~\cite{DK} combined \eqref{eq::trivialF} with the following very non-trivial lower bound on $e(\F)$:
\begin{theorem}[\cite{DK}] \label{thm::DK}
There exists $c > 0$ such that for all $n \leq 2k + k/6$ and $\F \in \T$,
\begin{equation}\label{eq::DK}
e(\F) > c \frac{1}{k} \binom{n-k-1}{k-1} a_\F \log\frac{\binom{n-1}{k}}{a_\F} .
\end{equation}
\end{theorem}

As we are interested in all $n$ and $k$, we need a lower bound on $e(\F)$ for not just $n \leq 2k+k/6$. We can remove the assumption on $n$ in Theorem~\ref{thm::DK} with a similar bound due to Das and Tran~\cite{DT}. It is implicit in their paper, so for completeness, we derive it from their removal lemma for $\T$.
\begin{theorem}[\cite{DT}]\label{thm::removal}
There is an absolute constant $D > 1$ such that if $n \geq 2k+1$ and $\F \in \T$ with $e(\F) < \beta\binom{n-1}{k-1}\binom{n-k-1}{k-1}$, where $\beta \leq \frac{n-2k}{(20D)^2n}$, then there is an $x \in [n]$ such that
\[ |\F \setminus \K_x| \leq D\beta \frac{n}{n-2k}\binom{n-1}{k-1} .\]
\end{theorem}

\begin{theorem}[\cite{DT}]\label{thm::DT}
There exists $c>0$ such that for all $n\geq 2k+1$ and $\F \in \T$,
\begin{equation} \label{eq::DT}
e(\F) \ge c \frac{n-2k}{n} \binom{n-k-1}{k-1} a_\F .
\end{equation}
\end{theorem}

\begin{proof}
Let $\F \in \T_x(a)$. Let $D$ be the constant given in Theorem~\ref{thm::removal}, and let
\[ \beta = \frac{a}{(20D)^2 \frac{n}{n-2k} \binom{n-1}{k-1}} \leq \frac{n-2k}{(20 D)^2 n} ,\]
which follows from $a = |\F\setminus \K_x| \leq |\F| = \binom{n-1}{k-1}$. Since $x_\F = x$, we have that for all stars $\K_y$,
\[ |\F \setminus \K_y| \geq |\F \setminus \K_x| = a > \frac{a}{400 D} = D\beta \frac{n}{n-2k} \binom{n-1}{k-1} .\]
Applying the contrapositive of Theorem~\ref{thm::removal} to $\F$ with the above $\beta$, we have
\[ e(\F) \geq \beta \binom{n-1}{k-1} \binom{n-k-1}{k-1} = \frac{1}{(20D)^2} \cdot \frac{n-2k}{n} \binom{n-k-1}{k-1} a ,\]
which proves the theorem with $c=1/(20D)^2$.
\end{proof}

For convenience, we combine Theorems~\ref{thm::DK} and~\ref{thm::DT} so that they work for all $n$ and $k$.
\begin{theorem}\label{thm::eF}
There exists an absolute constant $\theta>0$ such that for all $n \geq 2k+1$ and $\F \in \T$,
\begin{equation}\label{eq::eF:DK}
e(\F) \geq \theta \frac{1}{p_0} a_\F \logaF ,
\end{equation}
and
\begin{equation}\label{eq::eF:DT}
e(\F) \geq \theta \frac{1}{p_0} a_\F \frac{n-2k}{n} \log\binom{n-1}{k} .
\end{equation}
\end{theorem}

\begin{proof}
Note that, from \eqref{eq::p0}, we have
\[ p_0 \geq \frac{k\log\frac{n-1}{k}}{\binom{n-k-1}{k-1}} \geq \frac{k\log 2}{\binom{n-k-1}{k-1}} ,\]
so we can replace \eqref{eq::DK} with $e(\F) \geq c \frac{1}{p_0} a_\F \logaF$, albeit with a different constant $c$. For $2k+k/6 < n$, we have
\[ \frac{1}{p_0} a \loga \leq \binom{n-k-1}{k-1} a \frac{\log\binom{n-1}{k}}{\log\left(n\binom{n-1}{k}\right)} \leq \binom{n-k-1}{k-1} a \cdot 13 \frac{n-2k}{n} ,\]
so by Theorem~\ref{thm::DT}, $e(\F) \geq c \frac{1}{p_0} a_\F \logaF$ holds for all $n$ and $k$, again with a different constant $c$, yielding~\eqref{eq::eF:DK}. Equation~\eqref{eq::eF:DT} follows directly from Theorem~\ref{thm::DT} and~\eqref{eq::p0}.
\end{proof}
From here on, we fix $\theta$ sufficiently small given by Theorem~\ref{thm::eF}. Using this, we can repeat the argument of Devlin and Kahn~\cite{DK}, which is just upper bounding \eqref{eq::unionbound}, to show that we do not need to consider those $\F$ with $e(\F)$ large. Let
\[ \T^1 = \left\{\F \in \T : e(\F) \leq \frac{5}{p_0} a_\F \log\frac{\binom{n-1}{k}}{a_\F} \right\} .\]
\begin{lemma}\label{lem::bigeF}
Assume $k = \omega(1)$ and let $p \geq \frac{49}{50}p_0$. Then no $\F \in \T \setminus \T^1$ is independent in $K_p(n,k)$ w.h.p.
\end{lemma}

\begin{proof}
Let $x \in [n]$. By Observation~\ref{obs::abound}, for every $\F \in \T$, we have $\binom{n-1}{k}/a_\F \geq \frac{n}{k} > 2$. Hence by~\eqref{eq::trivialF}, we have
\[ |\T_x(a)| \leq \binom{\binom{n-1}{k}}{a}^2 \leq \exp\left( 2a \log\frac{e\binom{n-1}{k}}{a} \right) \leq \exp\left( \left( 2 + \frac{2}{\log 2} \right) a \loga \right) .\]
Using $(1-p)^{e(\F)} \leq \exp(-pe(\F))$ and the union bound, the probability that some $\F \in \T_x(a) \setminus \T^1_x(a) $ is independent is at most
\[ \exp\left( \left( 2 + \frac{2}{\log 2} \right)a \loga - \frac{49}{10} a \loga \right) \leq \exp\left( - \frac{1}{70} a \log\frac{\binom{n-1}{k}}{a} \right) .\]
Taking the union bound over all $a$, the probability that $\F \in \T \setminus \T^1$ with $x_\F = x$ is independent is at most
\[ \sum_{a=1}^{n} \exp\left( -\frac{1}{70} a \log\frac{\binom{n-1}{k}}{a} \right) + \sum_{a=n+1}^{\binom{n-1}{k}} \exp\left( -\frac{1}{70} a \log\frac{\binom{n-1}{k}}{a} \right) \]
\[ \leq n \exp\left( -\frac{1}{70} \log\binom{n-1}{k} \right) + \binom{n-1}{k}\exp\left( - \frac{1}{70} n \log\frac{\binom{n-1}{k}}{n} \right) \]
\[ \leq \exp(-\omega(\log n)) + \exp(-\omega(n)) = o(1/n) ,\]
since $\log n = o(\log\binom{n-1}{k})$ for $k = \omega(1)$. A union bound over all $x$ completes the proof.
\end{proof}
In the previous proof, we took a union bound over certain $\F$ with $a_\F = a$ and $x_\F = x$, and then we took a union bound over all $a$ and $x$. Our proofs often take this form, and since the latter union bound can often be executed as above, we omit these details going forward.

Note that if $\F \in \T^1$, we get the following upper bound on $a_\F$, which will be particularly useful when $n-2k$ is large.

\begin{lemma}\label{lem::logafbound}
For every $\F \in \T^1$, we have
\[ \frac{\theta}{5} \cdot \frac{n-2k}{n} \log\binom{n-1}{k} \leq \logaF .\]
\end{lemma}

\begin{proof}
For $\F \in \T^1$, by~\eqref{eq::eF:DT}, we have
\[ \theta \frac{1}{p_0} a_\F \frac{n-2k}{n} \log\binom{n-1}{k} \leq e(\F) \leq \frac{5}{p_0} a_\F \logaF .\qedhere\]
\end{proof}

\subsection{Neighborhood assumption} \label{sec::prelim:neigh}

Let $\T^2 = \left\{\F \in \T^1 : \B_\F \subseteq N(\A_\F) \right\}$. Our first reduction is that for most $\F \in \T^1$, there exists $\F' \in \T^2$ such that $E(\F') \subseteq E(\F)$. This is helpful because if we know that $\F'$ is not independent, then we can automatically conclude that $\F$ is not independent. To make this reduction, we need a lemma concerning the vertex boundary of $\F \in \T$ in $K(n,k)$, which will also be useful for the next reduction. We derive this from the Kruskal-Katona theorem --- in fact, we only need a weaker version due to Lov\'asz. For real $z$, we define $\binom{z}{k} = \frac{z(z-1) \cdots (z-k+1)}{k!}$.
\begin{theorem}[Kruskal~\cite{Kruskal}, Katona~\cite{Katona}, Lov\'asz (Problem 13.31(b) of~\cite{Lovasz})]\label{thm::KK}
Let $\F$ be a nonempty family of $k$-sets, and let $z \geq k$ be a real number such that $|\F| = \binom{z}{k}$. Then for every $\ell\leq k$,
\[ \left|\left\{ S : |S|=\ell, \exists F \in \F \text{ such that } S \subseteq F \right\}\right| \geq \binom{z}{\ell} .\]
\end{theorem}

\begin{lemma}\label{lem::iso}
Let $x \in [n]$. For every nonempty $\A \subseteq \binom{[n]\setminus\{x\}}{k}$ with $|\A| < \binom{n-2}{k-1}$, we have $|N(\A) \cap \K_x| > |\A|$. Moreover, for every $\F \in \T_x(a)$, we have $|N(\A_\F) \cap \K_x| > a$.
\end{lemma}

\begin{proof}
Consider $\A^\star = \{ ([n]\setminus\{x\}) \setminus A : A \in \A\} \subseteq \binom{[n]\setminus\{x\}}{n-k-1}$. Then $|\A^\star| = |\A|$ and
\begin{equation}\label{eq::shadow}
\left|\left\{S \in \binom{[n]\setminus\{x\}}{k-1} :  \exists A\in \A^\star \textrm{ such that } S \subseteq A\right\}\right| = \left|N(\A) \cap \K_x\right| .
\end{equation}
Since $|\A| < \binom{n-2}{n-k-1}$, there exists a real number $z < n-2$ such that $|\A| = \binom{z}{n-k-1}$ and $n-k-1 \leq z$. Applying Theorem~\ref{thm::KK} to $\A^\star$ and using \eqref{eq::shadow}, we have
\begin{equation}\label{eq::expansion}
|N(\A) \cap \K_x| \geq \binom{z}{k-1} > \binom{z}{n-k-1} = |\A| ,
\end{equation}
where the second inequality follows because $\binom{n-2}{k-1} = \binom{n-2}{n-k-1}$, $n-k-1>k-1$, and $z<n-2$.

To prove the second part of the lemma, note that by Observation~\ref{obs::abound}, for any $\F \in \T$,
\[ |\A_\F| \leq \frac{n-k}{n} \binom{n-1}{k-1} = \frac{n-1}{n} \binom{n-2}{k-1} ,\]
and we are done by applying the first part of the lemma.
\end{proof}

Now we can give the reduction from $\T^1$ to $\T^2$, which we employ in Sections~\ref{sec::hit:proof} and~\ref{sec::largen:proof}.

\begin{prop}\label{prop::T2reduc}
For every $\F \in \T^1_x(a)$ with $a < \frac{1}{3} \binom{n-2}{k-1}$, there exists $\F' \in \T^2_x(a)$ satisfying $\A_{\F'} = \A_\F$ and $E(\F') \subseteq E(\F)$.
\end{prop}

\begin{proof}
Let $\F \in \T_x(a)$. Since $a \leq |N(\A_\F) \cap \K_x|$ by Lemma~\ref{lem::iso}, there exists $\B'$ of size $a$ such that $N(\A_\F) \cap \B_\F \subseteq \B' \subseteq N(\A_\F) \cap \K_x$. Let $\F' = \A_\F \cup (\K_x \setminus \B')$, so
\[ E(\F') = E(\A_\F) \cup E(\A_\F, \bar{\B'}) \subseteq E(\A_\F) \cup E(\A_\F, \bar{N(\A_\F) \cap \B_\F}) = E(\F) .\]
By Observation~\ref{obs::farstar}, for any $y \in [n]$ with $y \neq x$,
\[ |\F' \setminus \K_y| \geq |\F \setminus \K_y| - |\B'| \geq \binom{n-2}{k-1} - 2a > a = |\F' \setminus \K_x| .\]
Thus $x_{\F'} = x$ and so $\F' \in \T^2_x(a)$ and $\A_{\F'} = \A_\F$.
\end{proof}

\section{Main lemma}\label{sec:main lemma}

One might hope that, with the extra condition $\T^1$ imposes over $\T$, we have much stronger upper bounds on $|\T^1_x(a)|$ than~\eqref{eq::trivialF}. This is the content of the following technical lemma, which is one of our main innovations.

Note that $\binom{n-k-1}{k-1}$ is the degree of vertices in $\binom{[n]\setminus\{x\}}{k}$ to $\K_x$, and note that $\binom{n-k}{k} = \frac{n-k}{k} \binom{n-k-1}{k-1}$ is the degree of vertices in $\K_x$ to $\binom{[n]\setminus\{x\}}{k}$. For each $\F$ and $\delta = \delta(n,k)$, define $\A_\F^{\geq\delta} = \{A \in \A_\F : d(A,\B_\F) \geq \delta \binom{n-k-1}{k-1}\}$, where $d(A,\B_\F) = |N(A) \cap \B_\F|$. Let $\A_\F^{<\delta} = \A_\F \setminus \A_\F^{\geq\delta}$. We use the structure between $\A_\F$ and $\B_\F$ for $\F \in \T^1$ to give upper bounds for the number of $\A_\F^{\geq\delta}$ across all $\F \in \T^1_x(a)$.

\begin{lemma}\label{lem::Ahigh}
If $\delta = \omega\left( \frac{n\log\frac{n}{n-2k}}{k\log\binom{n-1}{k}} \right)$ and $\delta \leq 1/2$, then the number of choices for $\A_\F^{\geq\delta}$ across all $\F \in \T^1_x(a)$ is at most $\eoaloga$.
\end{lemma}

\begin{proof}
Let $D = \binom{n-k-1}{k-1}$, and choose $p_1$ so that $p_1 = \omega(1/\delta D)$ and $\frac{p_1}{p_0} \frac{n}{k} \log\frac{n}{n-2k} = o(1)$. This is possible because the assumption on $\delta$ implies that
\[ \frac{1}{p_0} \frac{n}{k} \log\frac{n}{n-2k} = o(\delta D) .\]
To each $\A_\F^{\geq\delta}$ we give a certificate $(\Y,\A_1,\A_2,\A_3)$ with the following properties:
\begin{enumerate}
\item $\Y \subseteq \K_x$,\, $|\Y| \leq 3ap_1$,
\item $\A_1 \subseteq N(\Y)$,\, $|\A_1| \leq \frac{p_1}{p_0} \cdot \frac{30n}{\theta k} \cdot a \loga$,\, $|N(\Y) \setminus \A_1| \leq a$,
\item $\A_2 \subseteq N(\Y) \setminus \A_1$,\, $|\A_2| \leq \frac{10}{\log\binom{n-1}{k}} a \log\frac{\binom{n-1}{k}}{a}$,
\item $\A_3 \subseteq \binom{[n]\setminus\{x\}}{k}$,\, $|\A_3| = o(a)$,
\item $\A_\F^{\geq\delta} = (N(\Y)\setminus (\A_1 \cup \A_2)) \cup \A_3$.
\end{enumerate}
By the last property, we can reconstruct $\A_\F^{\geq\delta}$ from its certificate. We postpone counting the number of choices of possible certificates until after we prove the existence of such certificates.

Fix $\F \in \T^1_x(a)$, and let $\A^{\geq\delta} = \A^{\geq\delta}_\F$ and $\A^{<\delta} = \A^{<\delta}_\F$. We prove that there is a particular choice of $\Y$ for which the elements of the certificate are $\Y$, $\A_1 := N(\Y) \setminus \A$, $\A_2 := (N(\Y) \setminus \A_1) \cap \A^{<\delta}$, and $\A_3 := \A^{\geq\delta} \setminus N(\Y)$. Select a random subset $\Y$ of $\B$, including each vertex independently with probability $p_1$. Then the expected size of $\Y$ is
\[ \E|\Y| = ap_1 .\]
The probability that $A \in \A^{\geq\delta}$ is not in the neighborhood of $\Y$ is at most $(1-p_1)^{\delta D}$, so
\[ \E|\A^{\geq\delta} \setminus N(\Y)| \leq a(1-p_1)^{\delta D} \leq ae^{-p_1 \delta D} = a e^{-\omega(1)} = o(a) .\]
Note that by Observation~\ref{obs::AbarB}, Theorem~\ref{thm::eF}, and the definition of $\T^1$,
\[ \E|N(\Y)\setminus\A| \leq \E \sum_{v\in\Y} d(v,\bar\A) = \E \sum_{v\in\B} 1_{v\in\Y} d(v,\bar\A) = p_1 e(\bar\A,\B) \]
\[ = p_1 \left( \binom{n-k-1}{k} a + e(\A,\bar\B) \right) = p_1 \left( \frac{\log\left(n\binom{n-1}{k}\right)}{k} \cdot \frac{1}{p_0} a (n-2k) + e(\A,\bar\B) \right) \]
\[ \leq p_1 \left( \frac{n}{\theta k}\cdot\frac{\log\left(n\binom{n-1}{k}\right)}{\log\binom{n-1}{k}} + 1 \right) e(\F) \leq \frac{p_1}{p_0} \cdot \frac{10n}{\theta k} \cdot a\loga .\]
By Markov's inequality, there exists a choice of $\Y$ such that
\[ |\Y| \leq 3ap_1 ,\]
\[ |\A^{\geq\delta} \setminus N(\Y)| \leq o(a) ,\]
\[ |N(\Y) \setminus \A| \leq \frac{p_1}{p_0} \cdot \frac{30n}{\theta k} \cdot a\loga .\]
Since $N(\Y) \setminus \A_1 \subseteq A$ by definition, all that remains to be shown is the upper bound on $|\A_2|$. Observe that
\[ (1-\delta) D |\A^{<\delta}| \leq e(\A,\bar\B) \leq e(\F) \leq \frac{5}{p_0} a \log\frac{\binom{n-1}{k}}{a} ,\]
so since $\delta \leq 1/2$,
\[ |\A_2| \leq |\A^{<\delta}| \leq \frac{10}{\log\binom{n-1}{k}} a \log\frac{\binom{n-1}{k}}{a} .\]

We count the number of possible certificates by counting each part of the certificate sequentially. Since $p_1 = o(1)$, $|\Y| = o(a)$, so the number of choices for $\Y$ is at most
\[ \binom{\binom{n-1}{k}}{o(a)} \leq \exp\left( o\left( a \loga \right) \right) .\]
Note that
\[ |N(\Y)| \leq \binom{n-k}{k}|\Y| \leq 3ap_1 \binom{n-k}{k} ,\]
so given $\Y$, the number of choices for $\A_1$ is at most, using Lemma~\ref{lem::logafbound},
\[ \begin{pmatrix} 3ap_1 \binom{n-k}{k} \\ \leq \frac{p_1}{p_0} \cdot \frac{30n}{\theta k} \cdot a \loga \end{pmatrix} \leq \exp\left( O\left( \frac{p_1}{p_0} \cdot \frac{n}{k} \cdot a \loga \log\frac{p_0 \binom{n-k}{k}}{\frac{30n}{\theta k} \loga} \right) \right) \]
\[ \leq \exp\left( O\left( \frac{p_1}{p_0} \cdot \frac{n}{k} \cdot a \loga \log \frac{6(n-k)}{n-2k} \right) \right) = \exp\left( O\left( \frac{p_1}{p_0}\cdot \frac{n}{k} \log\frac{n}{n-2k} \right) a \loga \right) \]
\[ = \eoaloga .\]
Given $\Y$ and $\A_1$, the number of possible choices for $\A_2$ is at most, using Observation~\ref{obs::abound},
\[ \begin{pmatrix} a \\ \leq \frac{10}{\log\binom{n-1}{k}} a \log\frac{\binom{n-1}{k}}{a} \end{pmatrix} \leq \exp\left( o\left( a \loga \right) \right) .\]
Finally, since $|\A_3| = o(a)$, the number of choices for $\A_3$ is at most
\[ \binom{\binom{n-1}{k}}{o(a)} \leq \exp\left( o\left( a \loga \right) \right) .\qedhere\]
\end{proof}

Let $\B_\F^{\geq\delta} = \left\{B \in \B_\F : d(B, \A_\F) \geq \delta \binom{n-k}{k}\right\}$ and $\B_\F^{<\delta} = \B_\F \setminus \B_\F^{\geq\delta}$. There is a nearly identical lemma for $\B_\F^{\geq\delta}$ as well.
\begin{lemma}\label{lem::Bhigh}
If $\delta = \omega\left( \frac{k}{n} \frac{\log\log\binom{n-1}{k}}{\log\binom{n-1}{k}} \right)$ and $\delta \leq 1/2$, then the number of $\B_\F^{\geq\delta}$ across all $\F \in \T_x^1(a)$ is at most $\exp\left( o\left( a \log\frac{\binom{n-1}{k}}{a} \right) \right)$.
\end{lemma}
The proof is essentially the same as the proof of the lemma for $\A_\F^{\geq\delta}$, except that we use $D = \binom{n-k}{k}$, $\frac{p_1}{p_0} \log\log\binom{n-1}{k} = o(1)$, and for $\X \subseteq \A$ randomly chosen with probability $p_1$, we have
\[ \E|N(X)\setminus \B| \leq p_1 e(\A,\bar\B) \leq 5 \frac{p_1}{p_0} a \loga .\]

\subsection{Large diversity when $n-2k$ is small}\label{sec::hit:largea}

The following proposition demonstrates our most basic way to show that the $\F \in \T^1$ are not independent in $K_p(n,k)$ w.h.p.: we use Lemmas~\ref{lem::Ahigh} and~\ref{lem::Bhigh} to bound the number of choices for $\A_\F^{\geq\delta}$ and $\B_\F^{\geq\delta}$, while we bound the number of $\A_\F^{<\delta}$ and $\B_\F^{<\delta}$ by controlling the size of those sets with $e(\F)$. This trick works when $n-2k$ is small and $a_\F$ is large.

\begin{prop}\label{prop::large k large a}
For any constant $\ep>0$, when $n-2k = o(n)$, no $\F \in \T^1$ with $\logaF = o(n)$ is independent in $K_p(n,k)$ w.h.p., where $p = (1-\ep)p_0$.
\end{prop}

\begin{proof}
By Lemmas~\ref{lem::Ahigh} and~\ref{lem::Bhigh}, since
\[ \frac{n\log\frac{n}{n-2k}}{k\log\binom{n-1}{k}} = o(1) \quad\text{ and }\quad \frac{k}{n} \frac{\log\log\binom{n-1}{k}}{\log\binom{n-1}{k}} = o(1) ,\]
the number of distinct $\A_\F^{\geq 1/2}$ and $\B_\F^{\geq 1/2}$ across all $\F \in \T^1_x(a)$ is at most
\[ \exp\left( o\left( a\loga \right)\right) .\]
Note that since $\loga = o(n)$ and $\F \in \T^1$,
\[ \frac{1}{2} \binom{n-k-1}{k-1} |\A^{<1/2}| \leq e(\A,\bar\B) \leq e(\F) \leq \frac{5}{p_0} a \loga = o\left( \frac{1}{p_0} an \right) ,\]
and since $\log\binom{n-1}{k} = \Theta(n)$, we have $|\A^{<1/2}| = o(a)$. Thus the number of such $\A^{<1/2}$ is at most
\[ \binom{\binom{n-1}{k}}{o(a)} \leq \exp\left( o\left( a \loga \right) \right) .\]
We perform a similar calculation for $\B^{<1/2}$. Using Observation~\ref{obs::AbarB}, \eqref{eq::eF:DK}, and $n-2k = o(n)$,
\[ \frac{1}{2}\binom{n-k}{k}|\B^{<1/2}| \leq e(\bar\A,\B) = \binom{n-k-1}{k} a + e(\A,\bar\B) = \binom{n-k-1}{k} a + o\left( \frac{an}{p_0} \right) ,\]
so $|\B^{<1/2}| = o(a)$, and we can proceed as before.

Taking the union bound and using~\eqref{eq::eF:DK}, we have that the probability some $\F \in \T^1_x(a)$ is independent is at most
\[ \exp\left( o\left( a \loga \right) \right) \cdot (1-p)^{\theta \frac{1}{p_0} a \loga} \leq \exp\left( -\frac{\theta}{2} a \loga \right) .\]
We take a union bound over all possible choices for $a$ and $x$ to finish the proof.
\end{proof}

\section{Hitting time for $n=2k+1$}\label{sec::hit}

For the entirety of this section, let $n=2k+1$, fix a constant $\ep>0$ sufficiently small, and let $p$ and $p'$ be such that
\[ p = 1-4^{-(1-\ep)} < \frac{3}{4} < 1-4^{-(1+\ep)} = p' .\]
We prove Theorems~\ref{thm::hitting:indep} and~\ref{thm::hitting:EKR} in Section~\ref{sec::hit:proof} as quick corollaries of a characterization of independent sets in $K_{\tau_{\text{super}}}(n,k)$ that holds with high probability, which we prove in Section~\ref{sec::hit:reduc}. To get this characterization, we first show in Section~\ref{sec::hit:binom} that, in a rigorous sense, $K_p(n,k)$ is a subgraph of $K_{\tau_{\text{super}}}(n,k)$ w.h.p. In Section~\ref{sec::hit:2link}, we describe a method of efficiently counting the number of possible $\A_\F$ using components in an auxiliary graph. We apply this method in Section~\ref{sec::hit:smalla} to show that no `connected' $\F$ is independent in $K_p(n,k)$ w.h.p., and we remove the connectedness condition in Section~\ref{sec::hit:reduc}.

\subsection{Approximation by the binomial model}\label{sec::hit:binom}

Let $G$ be a graph. Recall the definition of a random subgraph process of $G$: take a permutation of $E(G)$ uniformly at random and consider the initial segments of this permutation as a sequence of random spanning subgraphs of $G$. It will be convenient for us to generalize this model to the \emph{continuous time random subgraph process}, see Section~1.1 of~\cite{JLR}, which also generalizes the $K_p(n,k)$ model, that is, the model where we include each edge independently with a fixed probability. Assign to each edge $e$ of $G$ a uniformly randomly chosen real number $f(e)$ between $0$ and $1$. The underlying probability space is now the set of functions $f: E(G) \to [0,1]$ with the uniform distribution. Let $G_p(f)$ be the spanning subgraph of $G$ with $E(G_p(f)) = \{e \in E(G) : f(e) \leq p\}$. For deterministic $p$, that is, $p$ which do not depend on $f$, $G_p$ is simply the random subgraph of $G$ chosen by including each edge with probability $p$. However, we can let $p$ be a random variable, depending on $f$ and thus the structure of the random subgraph. This relates this model to the random subgraph process: with probability $1$, $f$ is injective, so as we slowly increase $p$ from $0$ to $1$, we randomly add edges one by one. We can analogously define the hitting time $q_P(f)$ for a monotone property $P$ in the $G_p$ model to be the smallest $p \in [0,1]$ such that $G_p(f)$ satisfies property $P$. If $\tau_P$ is the hitting time for $P$ in the random subgraph process, then $G_{q_P} = G_{\tau_P}$ as distributions. Define
\[ q_1(f) = q_1(n,k,f) = \min\{p \in [0,1] : K_p(n,k,f) \text{ has no independent superstars}\} ,\]
\[ q_2(f) = q_2(n,k,f) = \min\{p \in [0,1] : K_p(n,k,f) \text{ has no independent near-stars}\} .\]
To prove Theorem~\ref{thm::hitting:indep}, it suffices to show that $\alpha(K_{q_1(f)}(n,k,f)) = \binom{n-1}{k-1}$ w.h.p., where the `with high probability' is taken with respect to the uniform distribution of $f: E(K(n,k)) \to [0,1]$. Likewise, to prove Theorem~\ref{thm::hitting:EKR}, it suffices to show that $K_{q_2(f)}(n,k,f)$ is EKR w.h.p. To this end, we first show that $p_0$ is a sharp threshold for both containing independent superstars and containing independent near-stars, thus giving the bounds $(1-\ep)p_0 \leq q_1(f) < q_2(f) \leq (1+\ep)p_0$ w.h.p., for every constant\footnote{In fact, $\ep$ can be taken much smaller than a constant, but for $n=2k+1$ constant $\ep$ is sufficient.} $\ep>0$. Proofs of the lower bound appear in~\cite{BNR} and~\cite{DT}: a second moment calculation works for all $n>2k+1$ in~\cite{BNR}, while the simple calculation from~\cite{DT} works for a smaller range including $n=2k+1$ just as effectively. We repeat the argument of~\cite{DT} here for $n=2k+1$ for completeness.

\begin{observation}\label{obs::thresh}
There exists an independent superstar in $K_p(2k+1,k)$ w.h.p., where $p = 1-4^{-(1-\ep)}$. That is, $p < q_1(f)$ w.h.p.
\end{observation}
\begin{proof}
Since the superstars based at $1$ are edge-disjoint, the probability that $K_p(n,k)$ contains no independent superstar based at $1$ is
\[ \left( 1 - \left( 1-p \right)^{\binom{n-k-1}{k-1}} \right)^{\binom{n-1}{k}} = \left( 1 - 4^{-(1-\ep)k} \right)^{\binom{2k}{k}} \]
\[ \leq \exp\left( -4^{-(1-\ep)k} \binom{2k}{k} \right) \leq \exp\left(- \frac{4^{\ep k}}{2\sqrt{k}}\right) \to 0 .\qedhere\]
\end{proof}

The following gives a proof of Theorem~\ref{thm::main} for $n=2k+1$ from Theorem~\ref{thm::hitting:EKR}.

\begin{observation}\label{obs::threshupper}
W.h.p., $K_{p'}(2k+1,k)$ has no independent near-stars, where $p' = 1-4^{-(1+\ep)}$. That is, $q_2(f) < p'$ w.h.p.
\end{observation}

\begin{proof}
We call a near-star $\{A\} \cup (\K_x \setminus \{B\})$, where $A \not\in \K_x$ and $B \in \K_x$, \emph{maximal} if $A$ and $B$ are disjoint. If a non-maximal near-star is independent, then it is contained in an independent superstar. Therefore if there are no independent superstars and no independent maximal near-stars, then there are no independent near-stars. Note that $\binom{n-k-1}{k-1} = k$ is the degree of vertices of $\binom{[n]\setminus\{x\}}{k}$ to $\K_x$. The probability $A \not\in \K_x$ forms an independent superstar or maximal near-star based at $x$ is
\[ (1-p')^k + kp'(1-p')^{k-1} \leq 4k4^{-(1+\ep)k} .\]
Thus the expected number of independent superstars and maximal near-stars is at most
\[ n\binom{n-1}{k} \cdot 4k4^{-(1+\ep)k} \leq k^2 4^{-\ep k} \to 0 ,\]
so by Markov's inequality $K_{p'}(n,k)$ has no independent near-stars w.h.p.
\end{proof}

Going forward, the dependence of the hitting times $q_1(f)$ and $q_2(f)$ and random graphs $K_p(n,k,f)$ on $f$ will be implicit.

\subsection{Connected components}\label{sec::hit:2link}

Similar to how we reduced from $\T^1$ to $\T^2$ in Section~\ref{sec::prelim:neigh}, we reduce to `connected components' to deal with $\F \in \T^2$ with $a_\F$ small. Our goal is to produce a subcollection $\T^3$ of $\T^2$ such that: for $\F \in \T^2$ with $a_\F$ small, there exists $\F' \in \T^3$ such that $E(\F') \subseteq E(\F)$, and the number of $\A_\F$ across all $\F \in \T^3$ is small. This reduction is helpful because it often reduces the number of $\F$ we take a union bound over.

We define $\T^3$ by a connectedness condition in an auxiliary graph. Let $J_x(n,k)$ be the auxiliary graph on $\binom{[n]\setminus\{x\}}{k}$ defined by $A$ is adjacent to $A'$ if and only if $N(A) \cap N(A') \cap \K_x \neq \emptyset$ in $K(n,k)$, or equivalently, $|A \cup A'| \leq n-k$.\footnote{The Johnson `scheme' is a collection of graphs $J(n,k,c)$, defined on $\binom{[n]}{k}$ with $U$ and $V$ adjacent if and only if $|U \setminus V| = c$ (see~\cite{Johnson}). The Kneser graph is $J(n,k,k)$, and $J(n,k,1)$ and $J(n,k,n-2k)$ were used to establish the main result of~\cite{DK}. The union of $J_x(n,k)$ over all $x$ is the union of $J(n,k,c)$ over all $c \leq n-2k$.} If $A$ is adjacent to $A'$ in $J_x(n,k)$, $A$ and $A'$ are said to be \emph{2-linked with respect to $x$}. For $\A' \subseteq \A \subseteq \binom{[n]\setminus\{x\}}{k}$, we say $\A'$ is a \emph{2-linked component of $\A$ with respect to $x$} if $\A'$ is a connected component of the subgraph of $J_x(n,k)$ induced by $\A$. We say $\F \subseteq \binom{[n]}{x}$ is \emph{2-linked with respect to $x$} if the subgraph of $J(n,k)$ induced by $\F \setminus \K_x$ is connected.

Let $\T^3 = \{\F \in \T^2 : \F \text{ is 2-linked with respect to } x_\F\}$. The extra condition that $\T^3$ imposes over $\T^2$ cuts down on the number of choices for $\A_\F$. To see this, we need the following lemma about counting induced connected subgraphs.

\begin{lemma}[\cite{GK} Lemma 2.1]\label{lem::treebuilding}
Let $G$ be a graph on $n$ vertices of maximum degree at most $d$. The number of $A \subseteq V(G)$ with $|A| = a$ such that $G[A]$ is connected is at most $n(ed)^a$.
\end{lemma}

Since $J_x(2k+1,k)$ is $k^2$-regular, we immediately get the following lemma.

\begin{lemma}\label{lem::2linked}
The number of $\A_\F$ with $\F \in \T^3_x(a)$ is at most $\binom{2k}{k} (ek^2)^a$ for $n=2k+1$.
\end{lemma}

\subsection{Small diversity and 2-linked $\F$}\label{sec::hit:smalla}

Proposition~\ref{prop::large k large a} already handles $\F$ with $a_\F$ large; the following proposition deals with when $a_\F$ is small. Here we bound the number of $\F$ by counting 2-linked $\F$ with $\B_\F \subseteq N(\A_\F)$. When $a_\F$ is very small, we use a trivial bound on $e(\F)$ instead of Theorem~\ref{thm::eF}. We extend this proposition to all $\F$ in $\T^2$ in Section~\ref{sec::hit:reduc}.

\begin{prop}\label{prop::large k small a}
No $\F \in \T^3$ with $\logaF = \omega(\log k)$ is independent in $K_p(n,k)$ except if $a_\F = 1$ w.h.p., where $p = 1-4^{-(1-\ep)}$ and $n=2k+1$.
\end{prop}

\begin{proof}
By Lemma~\ref{lem::2linked}, the number of 2-linked $\A_\F$ with $\F \in \T^3_x(a)$ is at most
\[ \binom{2k}{k} \exp\left( O(a\log k) \right) = \binom{2k}{k} \exp\left( o\left( a \loga \right) \right) ,\]
since $\loga = \omega(\log k)$. The number of choices for $\B \subseteq N(\A)$, given $\A$, is at most
\[ \binom{a\binom{n-k-1}{k-1}}{a} = \binom{ak}{a} \leq \exp\left( O\left( a \log k \right) \right) = \exp\left( o\left( a \loga \right) \right) .\]
Note that $e(\F) \geq a(\binom{n-k-1}{k-1} - a) \geq \frac{7}{8} a k$ for $a \leq k/8$, so the probability some $\F \in \T^3_x(a)$ with $a \leq k/8$ is independent is at most
\[ \binom{2k}{k} \eoaloga (1-p)^{\frac{7}{8}ak} \leq 4^k \eoaloga 4^{-(1-\ep)\frac{7}{8}ak} \leq 4^{-ak/2} ,\]
using $a\geq 2$. For $a\geq k/8$, we use \eqref{eq::eF:DK}, so the probability that some $\F$ is independent is at most
\[ \exp\left( \log\binom{2k}{k} + o\left(a\loga\right) - \theta a \loga \right) = \exp\left( -\frac{\theta}{2} a \loga \right) .\]
In both cases, we can take a union bound over all such $a$ and $x$ to finish the proof.
\end{proof}

\subsection{Reduction to 2-linked components}\label{sec::hit:reduc}

Now we extend Proposition~\ref{prop::large k small a} from $\T^3$ to $\T^2$ by reducing $\F \in \T^2$ to its components in $\T^3$. Combined with Proposition~\ref{prop::large k large a}, this gives a characterization of the independent sets of $K_{q_1}(n,k)$, of which Theorems~\ref{thm::hitting:indep} and~\ref{thm::hitting:EKR} are a quick corollary.

\begin{prop}\label{prop::hitreduc}
Let $H = K_{q_1}(n,k)$. W.h.p., the only $\F \in \T$ which are independent in $H$ have the following form: $\F = \{A_1,\dots,A_m\} \cup (\K_x \setminus \{B_1,\dots,B_m\})$ where $E_H(\F,\K_x) = \{A_1B_1, \dots, A_m B_m\}$ and $m < \frac{1}{4} \binom{n-2}{k-1}$. 
\end{prop}

\begin{proof}
Let $H = K_{q_1}(n,k)$. By Observation~\ref{obs::thresh}, $q_1>p$ w.h.p., so assume this is the case. Furthermore, assume that $H$ satisfies the conclusions of Lemma~\ref{lem::bigeF} and Propositions~\ref{prop::large k large a} and~\ref{prop::large k small a}. Let $\F \in \T$ be independent in $H$, so by Lemma~\ref{lem::bigeF}, $\F \in \T^1_x(a)$ for some $x$. By Proposition~\ref{prop::large k large a}, $a < \frac{1}{3}\binom{n-2}{k-1}$, so let $\F' \in \T^2_x(a)$ be a set family given by applying Proposition~\ref{prop::T2reduc} to $\F$.

Let $\A_1, \dots, \A_m$ be the connected components of $\A_\F$ in $J(n,k)$, and let $\B_i = \B_{\F'} \cap N(\A_i)$ for $i \in [m]$. Let $i \in [m]$ be such that $|\A_i| \geq |\B_i|$. We claim that, with the assumptions on $H$, $|\A_i| = |\B_i| = 1$. Since there are no edges in $J_x(n,k)$ between distinct $\A_i$, the $N(\A_i) \cap \K_x$ and thus the $\B_i$ are pairwise disjoint. Because $\F' \in \T^2$, we have
\begin{equation}\label{eq::ABsum}
\sum_{i=1}^m |\B_i| = |\B_{\F'}| = |\A_{\F'}| = \sum_{i=1}^m |\A_i| .
\end{equation}
With our claim that $|\A_i| \geq |\B_i|$ implies $|\A_i| = |\B_i| = 1$, this gives that $|\A_i| = |\B_i| = 1$ for all $i \in [m]$, Let $\A_i = \{A_i\}$ and $\B_i = \{B_i\}$. Since the $N(\A_i)$ are disjoint, we have that $\F' = \{A_1,\dots,A_m\} \cup (\K_x \setminus \{B_1,\dots,B_m\})$ where $E_H(\F',\K_x) = \{A_1B_1, \dots, A_m B_m\}$. By construction, $\A_\F = \A_{\F'}$, so unless $\B_{\F} = \B_{\F'}$, $\F$ will contain some edge $A_i B_i$. Thus $\F = \F'$, and $\F$ has the desired structure.

All that remains is to prove the claim. Since $|\A_i| \leq a < \binom{n-2}{k-1}$, by Lemma~\ref{lem::iso}, $|N(\A_i) \cap \K_x| \geq |\A_i|$, so there exists $\B'$ of size $|\A_i|$ such that $N(\A_i) \cap \B_{\F'} \subseteq \B' \subseteq N(\A_i) \cap \K_x$. Let $\F'' = \A_i \cup (\K_x \setminus \B')$, and observe that
\[ E(\F'') = E(\A_i) \cup E(\A_i,\bar{\B'}) \subseteq E(\A_\F) \cup E(\A_i, \bar{N(\A_i) \cap \B_{\F'}}) \subseteq E(\F') .\]
By Observation~\ref{obs::farstar}, for any $y \in [n]$ with $y \neq x$,
\[ |\F'' \setminus \K_y| \geq |\F' \setminus \K_y| - |\B'| - |\A_\F| \geq \binom{n-2}{k-1} - 3a > a \geq |\A_i| = |\F'' \setminus \K_x| ,\]
which follows since $a < \frac{1}{4} \binom{n-2}{k-1}$ by Proposition~\ref{prop::large k large a}. Thus $x_{\F''} = x$ and so $\A_{\F''} = \A_i$, $\B_{\F''} = \B'$, and $\F'' \in \T^3_x(|\A_i|)$. Since $H$ satisfies the conclusion of Proposition~\ref{prop::large k small a}, we must have $a_{\F''} = 1$, so $|\A_i| = 1$. If $|\B_i| = 0$, then $\A_i \cup \K_x$ is an independent superstar in $H$, which is forbidden. Thus $|\A_i| = |\B_i| = 1$.
\end{proof}

\subsection{Proof of Theorems~\ref{thm::hitting:indep} and~\ref{thm::hitting:EKR}}\label{sec::hit:proof}

Compiling the results of this section, we first prove the simpler Theorem~\ref{thm::hitting:EKR} and then Theorem~\ref{thm::hitting:indep}.

\begin{proof}[Proof of Theorem~\ref{thm::hitting:EKR}]
Let $H = K_{q_2}(n,k)$. By Proposition~\ref{prop::hitreduc}, since $q_1<q_2$, w.h.p.\ the only set families in $\T$ which are independent in $H$ are $\{A_1,\dots,A_m\} \cup (\K_x \setminus \{B_1,\dots,B_m\})$ where $E_H(\F,\K_x) = \{A_1B_1, \dots, A_m B_m\}$. But $\{A_1\} \cup (\K_x \setminus \{B_1\})$ is a near-star, which is not independent in $H$. Thus no $\F \in \T$ is independent in $H$ w.h.p., as desired. 
\end{proof}

\begin{proof}[Proof of Theorem~\ref{thm::hitting:indep}]
Let $H = K_{q_1}(n,k)$. Consider a set family $\F \subseteq V(H)$ of size $\binom{n-1}{k-1}+1$ which is independent in $H$. Let $x$ minimize $|\K_x \setminus \F|$, and let $B \in \K_x \cap \F$. By Proposition~\ref{prop::hitreduc}, w.h.p.\ $\F \setminus \{B\} = \{A_1, \dots, A_m\} \cup (\K_y \setminus \{B_1, \dots, B_m\})$, for some $y \in [n]$, where $E_H(\F\setminus\{B\},\K_y) = \{A_1 B_1, \dots, A_m B_m\}$. This means that $|\K_y \setminus \F| < \frac{1}{4} \binom{n-2}{k-1}$, so by Observation~\ref{obs::farstar}, $y=x$. Thus $B = B_i$ for some $i$, so $\F$ contains the edge $A_i B_i$, a contradiction.
\end{proof}

\paragraph{Remark.} It is straightforward to generalize the proofs in this section to all $k$ with $n-2k = o(n)$. The main difference is that the degrees in the auxiliary graph $J_x(n,k)$ increase, but not so dramatically when $n-2k = o(n)$. However, for say $n=4k$, the graph $J_x(n,k)$ becomes complete, and so Lemma~\ref{lem::treebuilding} no longer provides an efficient way to count the number of $\A_\F$. 

One may also obtain hitting time results for $k = o(n)$ via a different approach, namely a careful modification of the argument in~\cite{DT}, which is a slight variation on the argument in~\cite{BNR}. For $\F$ with small $a_\F$ one counts maximal independent sets in $K(n,k)$, while $\F$ with large $a_\F$ are already handled by Lemma~\ref{lem::bigeF} and~\eqref{eq::eF:DT}, which imply
\[ \loga > \frac{\theta}{5} \cdot \frac{n-2k}{n} \log\binom{n-1}{k} \sim k\log\frac{n}{k} .\]
When $k$ is very small, further methods from~\cite{BNR} must be used. One must also be careful with the choice of $\ep$, as a constant $\ep$ does not work with these techniques. Unfortunately, these techniques break down when $k = \Omega(n)$ and in that case only handle $\F$ with $a_\F = o(n)$.

\section{Sharp threshold for $n\leq Ck$}\label{sec::largen}

As Theorem~\ref{thm::main} is proven for $n=2k+1$ in Section~\ref{sec::hit}, and Das and Tran~\cite{DT} proved Theorem~\ref{thm::main} for $k \leq n/C$ for some constant $C$, we assume for this section that $2k+1 < n < Ck$, fixing this constant $C$ sufficiently large. Additionally fix a constant $\ep>0$ sufficiently small, and let $p = (1+\ep)p_0$.

In Section~\ref{sec::largen:mindeg}, we prove an important assumption about $K_{(1+\ep)p_0}(n,k)$ that we are unable to make about the hitting time version $K_{\tau_{\text{near}}}(n,k)$. Given this, in Section~\ref{sec::largen:Alow} we show that for every $\F$ in consideration, $|\A_\F^{<\delta}|$ is small, and so, with Lemma~\ref{lem::Ahigh}, the number of choices for $\A_\F$ is small. When $n-2k$ is small, this allows us to quickly finish in Section~\ref{sec::largen:largek}. Otherwise, this allows us to eliminate those $\F$ with $e(\A_\F)$ large from consideration in Section~\ref{sec::largen:eA}. Finally, in Section~\ref{sec::largen:B} we deal with counting $\B^{<\delta}$, and we wrap up the proof of Theorem~\ref{thm::main} in Section~\ref{sec::largen:proof}.

While our proof likely works for smaller $\ep$, giving tighter bounds on the `width of the window' for this threshold, we do not optimize our choice of $\ep$ here, mostly because we think the stronger hitting time version of Conjecture~\ref{conj::hitting} ought to be true.

\subsection{Minimum degree assumption}\label{sec::largen:mindeg}

First we show that not only are near-stars not independent in $K_p(n,k)$ w.h.p., but every superstar contains at least $\delta k$ edges in $K_p(n,k)$ for some constant $\delta>0$.
\begin{lemma} \label{lem::minDeg}
Let $H = K_p(n,k)$, where $p = (1+\ep)p_0$. There exists a constant $\delta = \delta(\ep) > 0$ such that, w.h.p., for every $x \in [n]$ and $A \in \binom{[n]\setminus\{x\}}{k}$, we have $d_H(A, \K_x) \geq \delta k$.
\end{lemma}

\begin{proof}
Let $x \in [n]$ and $S \in \binom{[n]\setminus\{x\}}{k}$. Let
\[ \delta = \frac{\ep^2(1+\ep)\log\left(n\binom{n-1}{k}\right)}{25k} ,\]
and note that $\delta$ is constant, since $\log\left(n\binom{n-1}{k}\right) = \Theta(k)$. Since $d_H(A,\K_x)$ is binomially distributed with mean
\[ p \binom{n-k-1}{k-1} = (1+\ep) \log\left(n\binom{n-1}{k}\right) = 25 \delta k/\ep^2 ,\]
by Lemma~\ref{lem::cher}, the probability that $d_H(A,\K_x)$ is less than $\delta k$ is at most
\[ \exp\left( -\left( 1-2\sqrt{\ep^2/25} \right) (1+\ep)\log\left(n\binom{n-1}{k}\right) \right) \leq \left(n\binom{n-1}{k}\right)^{-1-\ep/2} .\]
Taking the union bound over all $n\binom{n-1}{k}$ choices for $x$ and $A$, the lemma follows.
\end{proof}

We fix the $\delta$ given by Lemma~\ref{lem::minDeg} for the remainder of this section. The specific value of $\delta$ is not important; all that matters is that $\delta$ remains constant as $n \to \infty$.

\paragraph{Remark.} Lemma~\ref{lem::minDeg} already implies that $\F$ with $a_\F < \delta k$ are not independent in $K_p(n,k)$ w.h.p. This is because for such $\F$, $\B_\F$ is too small to absorb all the edges from a single vertex of $\A_\F$.

\subsection{Few vertices of low degree in $\A_\F$}\label{sec::largen:Alow}

Vertices of $\A_\F$ with low degree to $\B_\F$ in $K(n,k)$ will likely have relatively low degree to $\B_\F$ in $K_p(n,k)$. But if $K_p(n,k)$ satisfies the minimum degree assumption, then there must be edges between $\A_\F$ and $\bar\B_\F$, so $\F$ is not independent. We formalize this with the following lemma, which upper bounds the number of vertices of $\A_\F$ with low degree to $\B_\F$ and with Lemma~\ref{lem::Ahigh} upper bounds the total number of choices for $\A_\F$. Let
\[ \T^4 = \left\{\F \in \T^1 : \left|\A_\F^{<1/\sqrt{k}}\right| \leq \frac{a_\F}{\log\log k} \right\} .\]

\begin{prop}\label{prop:fewlowdeg}
No $\F \in \T^1 \setminus \T^4$ is independent in $K_p(n,k)$ w.h.p., where $p = (1+\ep)p_0$.
\end{prop}

\begin{proof}
Let $H = K_p(n,k)$, and consider $\F \in \T^1_x(a) \setminus \T^4_x(a)$. For $A \in \A_\F^{<1/\sqrt{k}}$, $d_H(A,\B_\F)$ is binomially distributed with mean $p d_{K(n,k)}(A,\B_\F) \leq (1+\ep) \frac{1}{\sqrt{k}} \log\binom{n-1}{k}$. By Lemma~\ref{lem::cher}, the probability that this degree is at least $\delta k$ is at most
\[ \exp\left( - \delta k \log\frac{\delta k}{e(1+\ep)\frac{1}{\sqrt{k}} \log\binom{n-1}{k}} \right) \leq \binom{n-1}{k}^{-3\log\log k} .\]
Since the edges incident between different $A \in \A_\F$ and $\B_\F$ are distinct, the probability that every $A \in \A_\F^{<1/\sqrt{k}}$ satisfies $d_H(A,\B_\F) \geq \delta k$ is at most $\binom{n-1}{k}^{-3a}$. If for some $A \in \A_\F$, $d_H(A,\B_\F) < \delta k$, but $d_H(A,\K_x) \geq \delta k$, then $\F$ is not independent. By~\eqref{eq::trivialF}, applying the union bound over all $\F$ yields that the probability that some $\F \in \T^1_x(a) \setminus \T^4_x(a)$ is independent, and $H$ satisfies the minimum degree assumption, is at most
\[ \binom{\binom{n-1}{k}}{a} \binom{\binom{n-1}{k-1}}{a} \binom{n-1}{k}^{-3a} \leq \binom{n-1}{k}^{-a} .\]
Taking a union bound over all $a$ and $x$ and applying Lemma~\ref{lem::minDeg} finishes the proof.
\end{proof}

Using the definition of $\T^4$, we can conclude the following.

\begin{lemma}\label{lem::numA}
The number of $\A_\F$ across all $\F \in \T^4_x(a)$ is at most $\exp\left( o\left( a \log\frac{\binom{n-1}{k}}{a} \right) \right)$.
\end{lemma}

\begin{proof}
By Lemma~\ref{lem::Ahigh}, since
\[ \frac{1}{\sqrt{k}} = \omega\left( \frac{n\log\frac{n}{n-2k}}{k\log\binom{n-1}{k}} \right) ,\]
the number of $\A_\F^{\geq 1/\sqrt{k}}$ across all $\F \in \T^1_x(a)$ is at most
\[ \exp\left( o\left( a \loga \right) \right) .\]
Since $\left|\A_\F^{<1/\sqrt{k}}\right| = o(a)$ for all $\F \in \T^4_x(a)$, the number of choices for $\A_\F^{<1/\sqrt{k}}$ is at most
\[ \binom{\binom{n-1}{k}}{o(a)} = \exp\left( o\left( a\loga \right) \right) .\]
Multiplying these counts finishes the proof.
\end{proof}

\subsection{When $n-2k$ is small}\label{sec::largen:largek}

When $n-2k=o(n)$, we can essentially repeat the proof of Proposition~\ref{prop::large k small a}, although the calculations are easier this time around.

\begin{proposition}\label{prop::large k}
When $n-2k = o(n)$, no $\F \in \T^2 \cap \T^4$ with $\loga = \omega\left(\log\binom{n-k-1}{n-2k}\right)$ is independent in $K_p(n,k)$, where $p = (1+\ep)p_0$
\end{proposition}

\begin{proof}
By Lemma~\ref{lem::numA}, the number of $\A_\F$ across all $\F \in \T^4_x(a)$ is at most
\[ \eoaloga .\]
Recall that for $\F \in \T^2$, $\B_\F \subseteq N(\A_\F)$. Given $\A$, the number of choices for $\B \subseteq N(\A)$ is at most
\[ \binom{a\binom{n-k-1}{k-1}}{a} = \binom{a\binom{n-k-1}{n-2k}}{a} \leq \exp\left( a \log\binom{n-k-1}{n-2k} \right) \leq \eoaloga .\]
Using~\eqref{eq::eF:DK}, the probability that some $\F \in \T^4_x(a)$ is independent is at most
\[ \exp\left( o\left( a \loga \right) - (1+\ep) \theta a \loga \right) \leq \exp\left( - \frac{\theta}{2} a \loga \right) .\]
Taking a union bound over all $a$ and $x$ finishes the proof.
\end{proof}

\subsection{Few disjoint pairs in $\A$}\label{sec::largen:eA}

In the next subsection, we need lower bounds on $e(\A,\bar\B) = e(\F) - e(\A)$. Unfortunately, lower bounds for just $e(\F)$ are given in Theorem~\ref{thm::eF}, so we supply an upper bound on $e(\A)$ here. Let
\[ \T^5 = \left\{\F \in \T^4 : e(\A_\F) \leq \frac{1}{2} e(\F)\right\} .\]
\begin{corollary}\label{cor::eA}
No $\F \in \T^4 \setminus \T^5$ is independent in $K_p(n,k)$ w.h.p., where $p=(1+\ep)p_0$.
\end{corollary}

\begin{proof}
By Lemma~\ref{lem::numA}, the number of $\A_\F$ across all $\F \in \T^4_x(a)$ is at most
\[ \exp\left( o\left( a \loga \right) \right) .\]
The probability that a particular $\A_\F$ is independent, where $\F \in \T^4 \setminus \T^5$, is $(1-p)^{e(\A)} \leq (1-p)^{\frac{1}{2}e(\F)}$, so the probability that some $\F \in \T^4_x(a) \setminus \T^5_x(a)$ is independent is at most, using~\eqref{eq::eF:DK},
\[ \exp\left( o\left( a \loga \right) - \frac{1}{2} \theta a \loga \right) .\]
A union bound over all choices of $a$ and $x$ finishes the proof.
\end{proof}

\subsection{When $n-2k$ is large}\label{sec::largen:B}

If the number of $\B_\F$ across all $\F \in \T^5_x(a)$ were small, then we could proceed as in Proposition~\ref{prop::large k}. Lemma~\ref{lem::Bhigh} guarantees few choices for $\B_\F^{\geq \eta}$ for reasonable $\eta$, so we could try to prove that there are few choices for $\B_\F^{<\eta}$. Unfortunately, the strategy we used for $\A_\F^{<\eta}$ will not work here, since although we can make a similar minimum degree assumption, there would be no analogue of Proposition~\ref{prop:fewlowdeg}, as $e_H(\F) = 0$ does not guarantee that $e_H(\bar\A,\B) = 0$. 

In fact, we are unable to have good control of $|\B^{<\eta}|$ for reasonable $\eta$. Instead, by taking a union bound over all $\A_\F$ and $\B_\F^{\geq\eta}$, we will show that w.h.p.\ there are many edges between $\A_\F$ and $\bar{\B_\F^{\geq \eta}}$ for every $\F$ --- too many to be absorbed by $\B_\F^{<\eta}$, so these $\F$ cannot be independent. 

\begin{prop}\label{prop::med k}
When $n-2k = \omega(\log n)$ and $k = \omega(1)$, no $\F \in \T^5$ is independent in $K_p(n,k)$ w.h.p., where $p=(1+\ep)p_0$.
\end{prop}

\begin{proof}
Let $H = K_p(n,k)$. Consider $\F \in \T_x^5(a)$. Define
\[ \eta = \frac{\theta e^{-81/\theta}}{16} \cdot \frac{k}{n-k} \cdot \frac{\log\frac{\binom{n-1}{k}}{a}}{\log\left(n\binom{n-1}{k}\right)} \geq \Omega\left( \frac{k}{n-k} \cdot \frac{n-2k}{n} \right) ,\]
which follows from Lemma~\ref{lem::logafbound}. By Lemma~\ref{lem::numA}, the number of $\A_\F$ across all $\F \in \T^5_x(a)$ is at most
\[ \exp\left( o\left( a \loga \right) \right) .\]
By Lemma~\ref{lem::Bhigh}, the number of distinct $\B_\F^{\geq \eta}$ across all $\F \in \T^3$ is at most
\[ \exp\left( o\left( a \loga \right) \right) ,\]
since
\[ \eta = \Omega\left( \frac{k}{n}\cdot \frac{n-2k}{n} \right) = \omega\left( \frac{k}{n} \frac{\log\log\binom{n-1}{k}}{\log\binom{n-1}{k}} \right) .\]
Thus the number of $(\A,\B^{\geq\eta})$ pairs is at most
\[ \exp\left( o\left(a\loga\right) \right) .\]
Observe that $e_H(\A,\bar{\B^{\geq\eta}})$ is binomially distributed with mean
\[ p e_{K(n,k)}(\A,\bar{\B^{\geq\eta}}) \geq \frac{1}{2} (1+\ep) p_0 e(\F) \geq \frac{\theta}{2} a \loga .\]
By Lemma~\ref{lem::cher}, the probability that $e_H(\A,\bar{\B^{\geq\eta}}) \leq \frac{\theta}{16} a \log\frac{\binom{n-1}{k}}{a}$ is at most
\[ \exp\left( - \Omega\left( a \log\frac{\binom{n-1}{k}}{a} \right) \right) .\]

We may take the union bound over all such $\F$, $a$, and $x$ to get that, w.h.p., $e_H(\A_\F,\bar{\B_\F^{\geq\eta}}) \geq \frac{\theta}{16} a \log\frac{\binom{n-1}{k}}{a}$ for all $\F \in \T^5$.

Similarly, observe that $e_H(\A,\B^{<\eta})$ is binomially distributed with mean
\[ p e_{K(n,k)}(\A,\B^{<\eta}) \leq (1+\ep)p_0 |\B^{<\eta}| \eta \binom{n-k}{k} \leq (1+\ep) \frac{\theta}{16} e^{-81/\theta} a \loga ,\]
where in the second inequality we simply used $|\B^{<\eta}| \leq a$. By Lemma~\ref{lem::cher}, the probability that $e_H(\A,\B^{<\eta}) \geq \frac{\theta}{16} a \log\frac{\binom{n-1}{k}}{a}$ is at most
\[ \exp\left( - 5 a \log\frac{\binom{n-1}{k}}{a} \right) .\]
Using \eqref{eq::trivialF}, we can take a union bound over all $\F \in \T^3_x(a)$, and then union bound over all $a$ and $x$. Thus w.h.p., $e_H(\A,\B^{<\eta}) < \frac{\theta}{16} a \log\frac{\binom{n-1}{k}}{a}$ for all $\F \in \T^5$. 

Thus w.h.p., for all $\F \in \T^5$,
\[ e_H(\A_\F,\bar\B_\F) \geq e_H(\A_\F,\bar{\B_\F^{\geq\eta}}) - e_H(\A_\F,\B_\F^{<\eta}) > 0 ,\]
so no such $\F$ is independent in $H$.
\end{proof}

\subsection{Proof of Theorem~\ref{thm::main}}\label{sec::largen:proof}

\begin{proof}[Proof of Theorem~\ref{thm::main}]
By Lemma~\ref{lem::bigeF}, the only $\F \in \T$ which are independent in $K_p(n,k)$ are in $\T^1$ w.h.p.

For $n-2k \leq \sqrt{n}$, say, Proposition~\ref{prop::large k large a} gives that the only $\F \in \T^1$ which are independent in $K_p(n,k)$ satisfy $\loga = o(n)$. For each such $\F$, there exists $\F' \in \T^2$ with $E(\F'') \subseteq E(\F)$ by Proposition~\ref{prop::T2reduc}. Since
\[ \log\binom{n-k-1}{n-2k} = O\left((n-2k)\log\frac{n}{n-2k}\right) = o(n) ,\]
Propositions~\ref{prop:fewlowdeg} and~\ref{prop::large k} show that $\F'$ is not independent w.h.p., completing the proof for $n-2k \leq \sqrt{n}$.

For $n-2k \geq \sqrt{n}$, Proposition~\ref{prop:fewlowdeg} gives that the only $\F \in \T^1$ which are independent in $K_p(n,k)$ are in $\T^4$. The theorem then follows from Corollary~\ref{cor::eA} and Proposition~\ref{prop::med k}.
\end{proof}


\end{document}